\pdfoutput=1
\documentclass[a4paper,leqno,11pt]{amsart} 
\usepackage{amscd,color}

\theoremstyle{plain}
\newtheorem{theorem}{Theorem}

\newtheorem{lemma}[theorem]{Lemma}
\newtheorem{proposition}[theorem]{Proposition}

\numberwithin{theorem}{section}

\theoremstyle{definition}

\theoremstyle{remark}
\newtheorem{remark}[theorem]{Remark}

\author{Masanobu Kaneko} 
\address{Faculty of Mathematics, 
Kyushu University, 
Motooka 744, Nishi-ku,
Fuku\-oka 819-0395, 
Japan}
\email{mkaneko@math.kyushu-u.ac.jp}

\author{Masato Kuwata}
\address{Faculty of Economics, 
Chuo University, 
742-1 Higashinakano, 
Hachioji-shi, Tokyo 192-0393, 
Japan}
\email{kuwata@tamacc.chuo-u.ac.jp}


\usepackage{graphicx}
\usepackage[psdextra]{hyperref} %  for pdflatex
\hypersetup{
    colorlinks=false,
    pdfborder={0 0 0},
}
\usepackage{mathrsfs}

\title[Bianchi's elliptic quintic curves]%
{Bianchi's elliptic quintic curves and several modular function fields of level ten}

\date{\today}

 \def\Z{\mathbf {Z}} \def\P{\mathbf {P}}
\def\Q{\mathbf {Q}}  \def\C{\mathbf {C}}

\def\uh{\mathcal{H}}

\newcommand{\SL}{\operatorname{SL}}
\newcommand{\mg}{\SL_2(\Z)}
\def\<{\langle}\def\>{\rangle}
\def\sltwo(#1,#2;#3,#4){\mathchoice%
{\left(\hskip -\arraycolsep%
\begin{array}{rr} #1 & #2 \\ #3 & #4 \end{array}%
\hskip -\arraycolsep\right)}
{\left(\begin{smallmatrix}#1 & #2 \\ #3 & #4 \end{smallmatrix}\right)}
{\left(\begin{smallmatrix}#1 & #2 \\ #3 & #4 \end{smallmatrix}\right)}
{\left(\begin{smallmatrix}#1 & #2 \\ #3 & #4 \end{smallmatrix}\right)}
}
\def\sltwoc(#1,#2;#3,#4){\mathchoice%
{\begin{pmatrix} #1 & #2 \\ #3 & #4 \end{pmatrix}}
{\left(\begin{smallmatrix}#1 & #2 \\ #3 & #4 \end{smallmatrix}\right)}
{\left(\begin{smallmatrix}#1 & #2 \\ #3 & #4 \end{smallmatrix}\right)}
{\left(\begin{smallmatrix}#1 & #2 \\ #3 & #4 \end{smallmatrix}\right)}
}

\makeatletter
\def\@fnsymbol#1{\ensuremath{\ifcase#1\or \dagger\or \ddagger\or
   \mathsection\or \mathparagraph\or \|\or \dagger\dagger
   \or \ddagger\ddagger \else\@ctrerr\fi}}\makeatother

\numberwithin{equation}{section}

\makeatother

\def\thin{{\hskip 1pt}}

\begin{document}

\begin{abstract}
We provide an explicit description of two torsion points on the classical Bianchi elliptic quintic curve
in terms of Ramanujan's functions.  As a byproduct, we describe generators and defining
equations of several modular function fields of level 10 using those functions.
\end{abstract}

\maketitle

\section{Introduction}

Bianchi's elliptic curve over the complex number field $\C$ is an elliptic quintic in $\P^{4}$ given by the set of five quadratic equations (see \cite{Bianchi}):
\begin{equation}\label{eq:Bianchi}
E_\phi :  \left\{\renewcommand{\arraystretch}{1.2} \begin{array}{l}
x_{0}^{2} + \phi\thin x_{2}x_{3}-\phi^{-1}x_{1}x_{4} = 0, \\
x_{1}^{2} + \phi\thin x_{3}x_{4}-\phi^{-1}x_{2}x_{0} = 0, \\
x_{2}^{2} + \phi\thin x_{4}x_{0}-\phi^{-1}x_{3}x_{1} = 0, \\
x_{3}^{2} + \phi\thin x_{0}x_{1}-\phi^{-1}x_{4}x_{2} = 0, \\
x_{4}^{2} + \phi\thin x_{1}x_{2}-\phi^{-1}x_{0}x_{3} = 0.
\end{array}\right.
\end{equation}
Here, $\phi$ is a parametrizing modular function of level 5. It is well-known that we can take as $\phi$ the renowned Rogers-Ramanujan function
\[ \phi(\tau)\,=\, q^{\frac15}\prod_{n=1}^\infty\frac{(1-q^{5n-1})(1-q^{5n-4})}{(1-q^{5n-2})(1-q^{5n-3})} 
\,=\,\cfrac{q^{\frac15}}{1+\cfrac{q}{1+\cfrac{q^2}{1+\cfrac{q^3}{\ \,\ddots}}}}\qquad(q=e^{2\pi i\tau}),  \]
which is a generator of the field of modular functions on the principal congruence subgroup $\Gamma(5)$ of the modular group $\mg$. 

The aim of this article is to show that we can give generators and defining
equations of modular function fields for several groups between $\Gamma_1(5)$ and $\Gamma(10)$ in a rather unified manner,
by using the two division points of the Bianchi elliptic curve. 
Here, for each integer $N$, we denote by $\Gamma(N), \Gamma_1(N)$ and $\Gamma_0(N)$ the standard congruence subgroups of level $N$:
\allowdisplaybreaks
\begin{alignat*}{2}
&\Gamma(N) &&=\left\{\left.\sltwo(a,b;c,d) \in SL_{2}(\Z)\,\right|\, a\equiv d\equiv 1,b\equiv c\equiv 0\bmod N\right\}, \\
&\Gamma_1(N) &&=\left\{\left.\sltwo(a,b;c,d) \in SL_{2}(\Z)\,\right|\, a\equiv d\equiv 1, c\equiv 0\bmod N\right\}, \\
&\Gamma_0(N) &&=\left\{\left.\sltwo(a,b;c,d) \in SL_{2}(\Z)\,\right|\, c\equiv 0\bmod N\right\}. \\
\end{alignat*}
Also, we define the congruence subgroup $\Gamma^{2}$ by
\[ 
\Gamma^{2}:=\left\{\!\left.\sltwo(a,b;c,d) \in SL_{2}(\Z) \right|
\sltwo(a,b;c,d)\!\equiv\!\sltwo(1,0;0,1),\sltwo(1,1;1,0),\sltwo(0,1;1,1) (\bmod 2)\!\right\},
\]
which will appear several times later.

\section{The Jacobi theta functions and the Bianchi curve}
In this section, we give an overview of how to deduce~\eqref{eq:Bianchi} using classical Jacobi's theta 
function identities\footnote{Instead of Jacobi's theta functions, Bianchi used Weierstrass's $\sigma$-functions, as Klein and Hurwitz did. 
Both are essentially the same.},
and describe addition formulas on the curve $E_\phi$.  Although most of the algebraic formulas are valid over a field of 
arbitrary characteristic, we restrict ourselves to the case of the complex number field. For a general theory of elliptic normal curves of arbitrary level, 
we refer the reader to~\cite{KKtheta}.

For real parameters $p,q$, we define the theta function $\theta_{(p,q)}(z,\tau)$ with characteristic $(p,q)$ by
$$\theta_{(p,q)}(z,\tau):=\sum_{n\in\Z}e^{\pi i(n+p)^2\tau +2\pi i(n+p)(z+q)},$$
where $z\in~\C$ and $\tau\in~\uh =$ complex upper half-plane.
Define
$$\theta_k(z,\tau)=i^{-1}\theta_{(\frac1{2}-\frac{k}5,\frac52)}(5z,5\tau)
=i^{-1}\sum_{n\in\Z}e^{5\pi i(n-\frac{k}5+\frac1{2})^2\tau
+10\pi i(n-\frac{k}5+\frac1{2})(z+\frac12)}.$$
We use the case $k=0,1,2,3,4$ and $k=\frac12,\frac32,\frac52,\frac72,\frac92.$ 
Note that $\theta_k(z,\tau)$ depends only on $k\bmod\, 5$:  $\theta_{k+5}(z,\tau)=\theta_k(z,\tau)$.  
The following formulas will often be used (see {\it e.g.}~\cite{MumfordTata}).
\begin{alignat*}{2}
&\theta_k(z+1)=(-1)^{2k+1}\theta_k(z),  &&\theta_k(z+\tau)=-e^{-5\pi i \tau-10\pi i z} \theta_k(z),\\
&\theta_k(z+\frac15)=-\zeta^{-k}\theta_k(z), &&\theta_k(z+\frac{\tau}{10})=-ie^{-\pi i\tau/20-\pi iz}\theta_{k-\frac12}(z),\\
&\theta_k(z+\frac{\tau}5)=-e^{-\pi i\tau/5-2\pi iz}\theta_{k-1}(z), \quad&& \theta_k(z+\frac{2\tau}5)=e^{-4\pi i\tau/5-4\pi iz}\theta_{k-2}(z),
\end{alignat*}
where $\zeta=e^{2\pi i/5}$. We also note
\[ \theta_k(-z)=(-1)^{2k+1}\theta_{-k}(z) \]  and thus
\begin{equation}\label{null}  
\theta_0(0)=0,\quad\theta_3(0)=-\theta_2(0),\quad \theta_4(0)=-\theta_1(0).
\end{equation}
The two functions $\theta_{0}(z,\tau)$ and $\theta_{\frac52}(z,\tau)$ are the most basic ones:
\begin{eqnarray*}
\theta_{0}(z,\tau)=\theta_{(\frac12,\frac52)}(5z,5\tau)&=&
\sum_{n\in\Z}i^{2n+1}e^{\pi i\frac14(2n+1)^2(5\tau)+\pi i(2n+1)(5z)}
\\
&=&\text{ Jacobi's } -\vartheta_1(5z,5\tau) \text{ in \cite{Jacobi:Theta}},\\
\theta_{\frac52}(z,\tau)=\theta_{(0,\frac52)}(5z,5\tau)&=&
\sum_{n\in\Z}(-1)^ne^{\pi in^2(5\tau)+2\pi in(5z)}\\
&=&\text{ Jacobi's  }\vartheta(5z,5\tau) \text{ in  \cite{Jacobi:Theta}}.
\end{eqnarray*}

Let us deduce the Bianchi equation satisfied by
$$[\theta_0(z,\tau),\theta_1(z,\tau),\theta_2(z,\tau),\theta_3(z,\tau),\theta_4(z,\tau)]$$
as well as the addition formula using the theta relations in Jacobi~\cite{Jacobi:Theta}.
We fix $\tau\in\uh$ and write $\theta_k(z,\tau)$ as $\theta_k(z)$.
We start with the ``main identity'' Jacobi \cite{Jacobi:Theta}'s (A)-(4):
\begin{eqnarray}\label{A4}
&&\theta_{\frac52}(w)\theta_{\frac52}(x)\theta_{\frac52}(y)
\theta_{\frac52}(z)-\theta_0(w)\theta_0(x)\theta_0(y)\theta_0(z)
\nonumber \\
&&\qquad\qquad=\theta_{\frac52}(w')\theta_{\frac52}(x')\theta_{\frac52}(y')
\theta_{\frac52}(z')-\theta_0(w')\theta_0(x')\theta_0(y')\theta_0(z'),
\end{eqnarray}
where
\begin{gather*}
w'=\frac12(w+x+y+z),\quad x'=\frac12(w+x-y-z),\\
y'=\frac12(w-x+y-z),\quad z'=\frac12(w-x-y+z).
\end{gather*}
(This is a direct consequence of the identity 
$w^2+x^2+y^2+z^2 = w'^2+x'^2+y'^2+z'^2$.)

By the change  $w\to w+\tau/5\ $ in \eqref{A4}   (fixing $x,y,z$), we have
$w'\to w'+\tau/10,\ x'\to x'+\tau/10,\ 
y'\to y'+\tau/10,\ z'\to z'+\tau/10$  and by the transformation formulas above we obtain
\begin{eqnarray}\label{eq2}
&&-\theta_{\frac32}(w)\theta_{\frac52}(x)\theta_{\frac52}(y)
\theta_{\frac52}(z)+\theta_4(w)\theta_0(x)\theta_0(y)\theta_0(z)
\nonumber \\
&&\qquad\qquad=\theta_2(w')\theta_2(x')\theta_2(y')\theta_2(z')-
\theta_{\frac92}(w')\theta_{\frac92}(x')\theta_{\frac92}(y')
\theta_{\frac92}(z').
\end{eqnarray}
Next we make  $w\to w+\tau/5,\,x\to x+\tau/5,\,y\to y+\tau/5,\,z\to z+\tau/5$ ($w'\to w'+2\tau/5$ and $x',y',z'$ are fixed) in  \eqref{A4}
and \eqref{eq2}  to obtain
\begin{eqnarray}\label{eq3}
&&\theta_{\frac32}(w)\theta_{\frac32}(x)\theta_{\frac32}(y)
\theta_{\frac32}(z)-\theta_4(w)\theta_4(x)\theta_4(y)\theta_4(z)
\nonumber \\
&&\qquad\qquad=\theta_{\frac12}(w')\theta_{\frac52}(x')\theta_{\frac52}(y')
\theta_{\frac52}(z')-\theta_3(w')\theta_0(x')\theta_0(y')\theta_0(z')
\end{eqnarray}
and \begin{eqnarray}\label{eq4}
&&-\theta_{\frac12}(w)\theta_{\frac32}(x)\theta_{\frac32}(y)
\theta_{\frac32}(z)+\theta_3(w)\theta_4(x)\theta_4(y)\theta_4(z)
\nonumber \\
&&\qquad\qquad=\theta_0(w')\theta_2(x')\theta_2(y')\theta_2(z')-
\theta_{\frac52}(w')\theta_{\frac92}(x')\theta_{\frac92}(y')
\theta_{\frac92}(z').
\end{eqnarray}
Applying the same transformation to these two, we get
\begin{eqnarray}\label{eq5}
&&\theta_{\frac12}(w)\theta_{\frac12}(x)\theta_{\frac12}(y)
\theta_{\frac12}(z)-\theta_3(w)\theta_3(x)\theta_3(y)\theta_3(z)
\nonumber \\
&&\qquad\qquad=\theta_{\frac72}(w')\theta_{\frac52}(x')\theta_{\frac52}(y')
\theta_{\frac52}(z')-\theta_1(w')\theta_0(x')\theta_0(y')\theta_0(z')
\end{eqnarray}
and \begin{eqnarray}\label{eq6}
&&-\theta_{\frac92}(w)\theta_{\frac12}(x)\theta_{\frac12}(y)
\theta_{\frac12}(z)+\theta_2(w)\theta_3(x)\theta_3(y)\theta_3(z)
\nonumber \\
&&\qquad\qquad=\theta_3(w')\theta_2(x')\theta_2(y')\theta_2(z')-
\theta_{\frac12}(w')\theta_{\frac92}(x')\theta_{\frac92}(y')
\theta_{\frac92}(z').
\end{eqnarray}
Then set  $w=-(x+y+z)$ in \eqref{eq2} and \eqref{eq6}, which means 
$w'=0$, $x'=-(y+z)$, $y'=-(z+x)$, $z'=-(x+y)$, to obtain
\begin{multline}\label{eq7}
-\theta_{\frac72}(x+y+z)\theta_{\frac52}(x)\theta_{\frac52}(y)
\theta_{\frac52}(z)-\theta_1(x+y+z)\theta_0(x)\theta_0(y)\theta_0(z)
 \\
=\theta_3(0)\theta_3(y+z)\theta_3(z+x)\theta_3(x+y)-
\theta_{\frac12}(0)\theta_{\frac12}(y+z)\theta_{\frac12}(z+x)
\theta_{\frac12}(x+y)
\end{multline}
and \begin{multline}\label{eq8}
\theta_{\frac12}(x+y+z)\theta_{\frac12}(x)\theta_{\frac12}(y)
\theta_{\frac12}(z)+\theta_3(x+y+z)\theta_3(x)\theta_3(y)\theta_3(z)
\\
=\theta_3(0)\theta_3(y+z)\theta_3(z+x)\theta_3(x+y)+
\theta_{\frac12}(0)\theta_{\frac12}(y+z)\theta_{\frac12}(z+x)
\theta_{\frac12}(x+y).
\end{multline}
Also set $w=x+y+z$ in \eqref{eq5}, which gives  $w'=x+y+z$, $x'=x$, $y'=y$, $z'=z$ and 
\begin{multline}\label{eq9}
\theta_{\frac12}(x+y+z)\theta_{\frac12}(x)\theta_{\frac12}(y)
\theta_{\frac12}(z)-\theta_3(x+y+z)\theta_3(x)\theta_3(y)\theta_3(z)
\\
=\theta_{\frac72}(x+y+z)\theta_{\frac52}(x)\theta_{\frac52}(y)
\theta_{\frac52}(z)-\theta_1(x+y+z)\theta_0(x)\theta_0(y)\theta_0(z).
\end{multline}
By $(\eqref{eq7}+\eqref{eq8}-\eqref{eq9})/2$, we obtain
\begin{multline}\label{eq10}
\theta_3(0)\theta_3(x+y)\theta_3(y+z)\theta_3(z+x)
\\
=\theta_3(x+y+z)\theta_3(x)\theta_3(y)\theta_3(z)-
\theta_1(x+y+z)\theta_0(x)\theta_0(y)\theta_0(z).
\end{multline}
We then set $z=-y$ in this to obtain
\begin{equation}\label{eq11}
\theta_3(0)^2\theta_3(x+y)\theta_3(x-y)
=\theta_1(x)\theta_0(x)\theta_0(y)^2-
\theta_3(x)^2\theta_2(y)\theta_3(y).
\end{equation}
And further applying $x\to x+\tau/5,\,y\to y+\tau/5$ in turn, we have
\begin{align}
\theta_3(0)^2\theta_1(x+y)\theta_3(x-y)
&=\theta_0(x)\theta_4(x)\theta_4(y)^2-
\theta_2(x)^2\theta_1(y)\theta_2(y),\label{eq12}\\
\theta_3(0)^2\theta_4(x+y)\theta_3(x-y)
&=\theta_4(x)\theta_3(x)\theta_3(y)^2-
\theta_1(x)^2\theta_0(y)\theta_1(y),\label{eq13}\\
\theta_3(0)^2\theta_2(x+y)\theta_3(x-y)
&=\theta_3(x)\theta_2(x)\theta_2(y)^2-
\theta_0(x)^2\theta_4(y)\theta_0(y),\label{eq14}\\
\theta_3(0)^2\theta_0(x+y)\theta_3(x-y)
&=\theta_2(x)\theta_1(x)\theta_1(y)^2-
\theta_4(x)^2\theta_3(y)\theta_4(y).\label{eq15}
\end{align}
Changing $x\to x+\tau/5$, we obtain in total $25$ addition formulas 
\begin{align}
\theta_3(0)^2\theta_2(x+y)\theta_2(x-y)
&=\theta_0(x)\theta_4(x)\theta_0(y)^2-
\theta_2(x)^2\theta_2(y)\theta_3(y),\label{eq16}\\
\theta_3(0)^2\theta_0(x+y)\theta_2(x-y)
&=\theta_4(x)\theta_3(x)\theta_4(y)^2-
\theta_1(x)^2\theta_1(y)\theta_2(y),\label{eq17}\\
\theta_3(0)^2\theta_3(x+y)\theta_2(x-y)
&=\theta_3(x)\theta_2(x)\theta_3(y)^2-
\theta_0(x)^2\theta_0(y)\theta_1(y),\label{eq18}\\
\theta_3(0)^2\theta_1(x+y)\theta_2(x-y)
&=\theta_2(x)\theta_1(x)\theta_2(y)^2-
\theta_4(x)^2\theta_4(y)\theta_0(y),\label{eq19}\\
\theta_3(0)^2\theta_4(x+y)\theta_2(x-y)
&=\theta_1(x)\theta_0(x)\theta_1(y)^2-
\theta_3(x)^2\theta_3(y)\theta_4(y),\label{eq20}
\end{align}
\begin{align}
\theta_3(0)^2\theta_1(x+y)\theta_1(x-y)
&=\theta_4(x)\theta_3(x)\theta_0(y)^2-
\theta_1(x)^2\theta_2(y)\theta_3(y),\label{eq21}\\
\theta_3(0)^2\theta_4(x+y)\theta_1(x-y)
&=\theta_3(x)\theta_2(x)\theta_4(y)^2-
\theta_0(x)^2\theta_1(y)\theta_2(y),\label{eq22}\\
\theta_3(0)^2\theta_2(x+y)\theta_1(x-y)
&=\theta_2(x)\theta_1(x)\theta_3(y)^2-
\theta_4(x)^2\theta_0(y)\theta_1(y),\label{eq23}\\
\theta_3(0)^2\theta_0(x+y)\theta_1(x-y)
&=\theta_1(x)\theta_0(x)\theta_2(y)^2-
\theta_3(x)^2\theta_4(y)\theta_0(y),\label{eq24}\\
\theta_3(0)^2\theta_3(x+y)\theta_1(x-y)
&=\theta_0(x)\theta_4(x)\theta_1(y)^2-
\theta_2(x)^2\theta_3(y)\theta_4(y),\label{eq25}
\end{align}
\begin{align}
\theta_3(0)^2\theta_0(x+y)\theta_0(x-y)
&=\theta_3(x)\theta_2(x)\theta_0(y)^2-
\theta_0(x)^2\theta_2(y)\theta_3(y),\label{eq26}\\
\theta_3(0)^2\theta_3(x+y)\theta_0(x-y)
&=\theta_2(x)\theta_1(x)\theta_4(y)^2-
\theta_4(x)^2\theta_1(y)\theta_2(y),\label{eq27}\\
\theta_3(0)^2\theta_1(x+y)\theta_0(x-y)
&=\theta_1(x)\theta_0(x)\theta_3(y)^2-
\theta_3(x)^2\theta_0(y)\theta_1(y),\label{eq28}\\
\theta_3(0)^2\theta_4(x+y)\theta_0(x-y)
&=\theta_0(x)\theta_4(x)\theta_2(y)^2-
\theta_2(x)^2\theta_4(y)\theta_0(y),\label{eq29}\\
\theta_3(0)^2\theta_2(x+y)\theta_0(x-y)
&=\theta_4(x)\theta_3(x)\theta_1(y)^2-
\theta_1(x)^2\theta_3(y)\theta_4(y),\label{eq30}
\end{align}
\begin{align}
\theta_3(0)^2\theta_4(x+y)\theta_4(x-y)
&=\theta_2(x)\theta_1(x)\theta_0(y)^2-
\theta_4(x)^2\theta_2(y)\theta_3(y),\label{eq31}\\
\theta_3(0)^2\theta_2(x+y)\theta_4(x-y)
&=\theta_1(x)\theta_0(x)\theta_4(y)^2-
\theta_3(x)^2\theta_1(y)\theta_2(y),\label{eq32}\\
\theta_3(0)^2\theta_0(x+y)\theta_4(x-y)
&=\theta_0(x)\theta_4(x)\theta_3(y)^2-
\theta_2(x)^2\theta_0(y)\theta_1(y),\label{eq33}\\
\theta_3(0)^2\theta_3(x+y)\theta_4(x-y)
&=\theta_4(x)\theta_3(x)\theta_2(y)^2-
\theta_1(x)^2\theta_4(y)\theta_0(y),\label{eq34}\\
\theta_3(0)^2\theta_1(x+y)\theta_4(x-y)
&=\theta_3(x)\theta_2(x)\theta_1(y)^2-
\theta_0(x)^2\theta_3(y)\theta_4(y).\label{eq35}
\end{align}
By setting $x=y=z$, we obtain the duplication formula from \eqref{eq11} to  \eqref{eq15}.
\begin{align*}
\theta_3(0)^3\theta_0(2z)
&=\theta_2(z)\theta_1(z)^3-
\theta_4(z)^3\theta_3(z),\\
\theta_3(0)^3\theta_1(2z)
&=\theta_0(z)\theta_4(z)^3-
\theta_2(z)^3\theta_1(z),\\
\theta_3(0)^3\theta_2(2z)
&=\theta_3(z)\theta_2(z)^3-
\theta_0(z)^3\theta_4(z),\\
\theta_3(0)^3\theta_3(2z)
&=\theta_1(z)\theta_0(z)^3-
\theta_3(z)^3\theta_2(z),\\
\theta_3(0)^3\theta_4(2z)
&=\theta_4(z)\theta_3(z)^3-
\theta_1(z)^3\theta_0(z).
\end{align*}
If we understand the suffix modulo 5, these can be written uniformly as 
\[ \theta_2(0)^3\theta_k(2z)
=\theta_{3k+2}(z)\theta_{3k+1}(z)^3-
\theta_{3k-1}(z)^3\theta_{3k-2}(z).\] 
From  \eqref{eq21} to  \eqref{eq25}, we have
\begin{eqnarray*}
\theta_3(0)^2\theta_1(0)\theta_0(2z)
&=&\theta_0(z)\theta_1(z)\theta_2(z)^2-
\theta_0(z)\theta_4(z)\theta_3(z)^2,\\
\theta_3(0)^2\theta_1(0)\theta_1(2z)
&=&\theta_3(z)\theta_4(z)\theta_0(z)^2-
\theta_3(z)\theta_2(z)\theta_1(z)^2,\\
\theta_3(0)^2\theta_1(0)\theta_2(2z)
&=&\theta_1(z)\theta_2(z)\theta_3(z)^2-
\theta_1(z)\theta_0(z)\theta_4(z)^2,\\
\theta_3(0)^2\theta_1(0)\theta_3(2z)
&=&\theta_4(z)\theta_0(z)\theta_1(z)^2-
\theta_4(z)\theta_3(z)\theta_2(z)^2,\\
\theta_3(0)^2\theta_1(0)\theta_4(2z)
&=&\theta_2(z)\theta_3(z)\theta_4(z)^2-
\theta_2(z)\theta_1(z)\theta_0(z)^2,
\end{eqnarray*}
or simply
\[ \theta_3(0)^2\theta_1(0)\theta_k(2z)
=\theta_{3k}(z)\theta_{3k+1}(z)\theta_{3k+2}(z)^2-
\theta_{3k}(z)\theta_{3k-1}(z)\theta_{3k-2}(z)^2 .\]

Setting $y=0$ in the addition formulas \eqref{eq11} to  \eqref{eq35}, we obtain (changing $x$ into $z$)
\begin{eqnarray*}
\theta_3(0)^2\theta_1(z)\theta_4(z)
&=&\theta_4(0)^2\theta_2(z)\theta_3(z)-
\theta_1(0)\theta_2(0)\theta_0(z)^2,\\
\theta_3(0)^2\theta_2(z)\theta_0(z)
&=&\theta_4(0)^2\theta_3(z)\theta_4(z)-
\theta_1(0)\theta_2(0)\theta_1(z)^2,\\
\theta_3(0)^2\theta_3(z)\theta_1(z)
&=&\theta_4(0)^2\theta_4(z)\theta_0(z)-
\theta_1(0)\theta_2(0)\theta_2(z)^2,\\
\theta_3(0)^2\theta_4(z)\theta_2(z)
&=&\theta_1(0)^2\theta_0(z)\theta_1(z)-
\theta_3(0)\theta_4(0)\theta_3(z)^2,\\
\theta_3(0)^2\theta_0(z)\theta_3(z)
&=&\theta_1(0)^2\theta_1(z)\theta_2(z)-
\theta_3(0)\theta_4(0)\theta_4(z)^2.
\end{eqnarray*}

These are exactly Bianchi's equation. Namely, dividing all by  $\theta_1(0)\theta_2(0)$,
noting $\theta_3(0)=-\theta_2(0)$, $\theta_4(0)=-\theta_1(0)$, and setting
\[ \phi(\tau):=-\frac{\theta_1(0)}{\theta_2(0)}=q^{\frac15} - q^{\frac65} + q^{\frac{11}5} - q^{\frac{21}5} + q^{\frac{26}5} - q^{\frac{31}5} + q^{
 \frac{36}5} - q^{\frac{46}5} + 2 q^{\frac{51}5} -\cdots, \] we have
\begin{align*}
&\theta_0(z)^2+\phi(\tau)\theta_2(z)\theta_3(z)
-\frac1{\phi(\tau)}\theta_1(z)\theta_4(z)=0,\\
&\theta_1(z)^2+\phi(\tau)\theta_3(z)\theta_4(z)
-\frac1{\phi(\tau)}\theta_2(z)\theta_0(z)=0,\\
&\theta_2(z)^2+\phi(\tau)\theta_4(z)\theta_0(z)
-\frac1{\phi(\tau)}\theta_3(z)\theta_1(z)=0,\\
&\theta_3(z)^2+\phi(\tau)\theta_0(z)\theta_1(z)
-\frac1{\phi(\tau)}\theta_4(z)\theta_2(z)=0,\\
&\theta_4(z)^2+\phi(\tau)\theta_1(z)\theta_2(z)
-\frac1{\phi(\tau)}\theta_0(z)\theta_3(z)=0,
\end{align*}
or 
\[ \theta_k(z)^2+\phi(\tau)\theta_{k+2}(z)\theta_{k-2}(z)
-\frac1{\phi(\tau)}\theta_{k+1}(z)\theta_{k-1}(z)=0\quad (k=0,1,2,3,4). \]

The map $\C\to\P^4(\C)$ given by 
\[
z\mapsto [x_0:x_1:x_2:x_3:x_4]=[\theta_0(z):\theta_1(z):\theta_2(z):\theta_3(z):\theta_4(z)] 
\]
induces an embedding of the complex torus $L_\tau:=\C/(\Z+\Z\tau)$ onto an algebraic curve in $\P^4$, and this 
becomes an elliptic curve, as reviewed in \cite{KKtheta}.
The equation is
\[
E_{\phi} : 
\left\{\renewcommand{\arraystretch}{1.3}\begin{array}{l}
x_{0}^{2} + \phi x_{2}x_{3}-\phi^{-1}x_{1}x_{4} = 0, \\
x_{1}^{2} + \phi x_{3}x_{4}-\phi^{-1}x_{2}x_{0} = 0, \\
x_{2}^{2} + \phi x_{4}x_{0}-\phi^{-1}x_{3}x_{1} = 0, \\
x_{3}^{2} + \phi x_{0}x_{1}-\phi^{-1}x_{4}x_{2} = 0, \\
x_{4}^{2} + \phi x_{1}x_{2}-\phi^{-1}x_{0}x_{3} = 0.
\end{array}\right.
\]

In terms of the coordinates, the addition formulas  \eqref{eq26} to \eqref{eq30} read (we write the addition of points simply by $+$)
\begin{equation}\label{eq:add}
 [x_0:x_1:x_2:x_3:x_4]+[y_0: y_1: y_2: y_3: y_4]=[z_0: z_1: z_2: z_3: z_4], 
\end{equation}
where
\[
{\rm A}_1 : 
\left\{\renewcommand{\arraystretch}{1.3}\begin{array}{l}
z_0=x_2x_3y_0^2-x_0^2y_2y_3, \\
z_1=x_0x_1y_3^2-x_3^2y_0y_1, \\
z_2=x_3x_4y_1^2-x_1^2y_3y_4, \\
z_3=x_1x_2y_4^2-x_4^2y_1y_2, \\
z_4=x_4x_0y_2^2-x_2^2y_4y_0.
\end{array}\right.
\]
Note that the neutral element of $E_\phi$ corresponds to $z=0$ and so by~\eqref{null} and $\phi(\tau)=-\theta_1(0)/\theta_2(0)$
it is 
\[ O=[0:\phi:-1:1:-\phi]. \]

This addition formulas $A_1$ cannot be used if $\theta_0(x-y)=0$. This is the case when  $x-y=m/5\ (0\le m\le4)$ (see \cite{KKtheta}).
By $\theta_k(z+1/5)=-\zeta^{-k}\theta_k(z)$ ($\zeta=e^{2\pi i/5}$), this happens when
\begin{multline}\label{eq:bad}
 [y_0: y_1: y_2: y_3: y_4]
 \\
 =[x_0: \zeta^{-m}x_1: \zeta^{-2m}x_2: \zeta^{-3m}x_3: \zeta^{-4m}x_4] \ \ (0\le m\le4). 
\end{multline}
Since there are no common zeros among $\theta_k(z)$, in this case we can use the following which can be derived from 
\eqref{eq21} to \eqref{eq25}:
\[
{\rm A}_2 : 
\left\{\renewcommand{\arraystretch}{1.3}\begin{array}{l}
z_0=x_1x_0y_2^2-x_3^2y_0y_4, \\
z_1=x_4x_3y_0^2-x_1^2y_3y_2, \\
z_2=x_2x_1y_3^2-x_4^2y_1y_0, \\
z_3=x_0x_4y_1^2-x_2^2y_4y_3, \\
z_4=x_3x_2y_4^2-x_0^2y_2y_1.
\end{array}\right.
\]

\section{Weierstrass form of the elliptic curve $E_\phi$}

Bianchi \cite{Bianchi} showed that the curve $E_\phi$ is birationally equivalent to the following quintic curve:
\begin{equation}\label{eq:quintic}
C_{\phi}:\phi^6 x_{0}^5 + \phi x_{1}^5 + \phi^6 x_{2}^5+ \phi^4(\phi^5+3 ) x_{0}^2 x_{1} x_{2}^2
-(2 \phi^5 + 1) x_{0} x_{2} x_{1}^3 =0. 
\end{equation}
Indeed, this can be done by eliminating $x_{3}$ and $x_{4}$ using the first and the third equations of \eqref{eq:Bianchi}.  The plane curve defined by \eqref{eq:quintic} has five ordinary double points at $(1:-\phi^{2}\zeta_{5}^{k}:1)$ ($1\le k \le 5$).

The intersection between Bianchi's elliptic curve \eqref{eq:Bianchi} and each of the hyperplane $x_{k}=0$ ($1\le k \le 5$) consists of five points defined over $\Q(\zeta_{5}, \phi)$.  
These are the $5$-torsion points. We have chosen $O=(0:\phi:-1:1:-\phi)$ as the origin of the group structure.  Its image in the plane model \eqref{eq:quintic} is $(\phi:-1:1)$.  Using this point and van Hoeij's algorithm\cite{vanHoeij} we can convert \eqref{eq:quintic} to the Weierstrass form of $E_\phi$.

\begin{proposition}\label{prop:BianchiWeier}  The elliptic curve $E_\phi$ is birationally equivalent to the Weierstrass model
\begin{align*}  
W_{\phi}:
Y^2 &= X^3-\frac1{48}(\phi^{20} -228 \phi^{15} + 494\phi^{10}+228 \phi^5+1) X 
\\
 &\qquad  +\frac1{864} (\phi^{30} +522 \phi^{25} - 10005\phi^{20}-10005 \phi^{10} -522\phi^5+1),
\end{align*}
where the rational map $E_{\phi}\mapstochar\dashrightarrow W_{\phi}$ is given by
\begin{align*} 
X&=\frac{1}{12}(\phi^{10} + 30\phi^5 + 1) 
 - \phi^2(2\phi^5 + 1)\frac{x_{1}+x_{4}}{x_{0}}- \phi^3(\phi^5 - 2)\frac{x_{2}+x_{3}}{x_{0}} \\ 
 &\qquad\quad  - 5\phi^3\,\frac{x_{2}}{x_{0}}+5\phi^4\,\frac{x_{1}(x_{1}-\phi x_{2}+ \phi x_{3}- x_{4})}{x_{0}^2}+5\phi^5\,\frac{x_{2}x_{4}}{x_{0}^2},\\  
\intertext{and}
Y &= \frac1\phi\,\frac{(\phi^{11} +\phi^6-\phi)^2(x_2-x_3)}{(7-2\phi^5)\phi^3 x_0 + (7\phi^5 +1) (x_1+x_4) 
+ (3-4 \phi^5)\phi (x_2+x_{3})}\\ 
&=\frac12\frac{(\phi^{11} + \phi^6 - \phi)^2 (x_{1} - x_{4}) }{(7 \phi^5 +1)x_0+(3 \phi^5+4)\phi^2 (x_{1}+x_4) -(\phi^5-7)\phi^3 (x_{2} +x_{3})}. 
\end{align*}

The discriminant of the curve is 
\begin{align*} 
&\frac1{1728}\left((\phi^{20} -228 \phi^{15} + 494\phi^{10}+228 \phi^5+1)^3\right. \\
& \qquad\qquad -\left.(\phi^{30} +522 \phi^{25} - 10005\phi^{20}-10005 \phi^{10} -522\phi^5+1)^2\right)\\
&=\phi^{5}(\phi^{10}-11\phi^5+1)^5.
\end{align*}
\end{proposition}

\begin{remark}\label{rmk:3.1-inverse}
1)  The inverse rational map $W_{\phi}\mapstochar\dashrightarrow E_{\phi}$ is given by
\[
x_{1}=\frac{f_{1}}{\phi^{2}d}, \ 
x_{2}=\frac{f_{2}}{\phi^{3}d}, \ 
x_{3} = -\frac{\phi (\phi^{2} x_{0}^3 + x_{1} x_{2}^2)}{\phi^{4} x_{0} x_{2} - x_{1}^2},
\quad
x_{4} = -\frac{\phi (\phi^{2} x_{2}^3 + x_{0}^2 x_{1})}{\phi^{4} x_{0} x_{2} - x_{1}^2},
\]
where
\begin{align*}
f_{1}&=
-144 (\phi^{5} + 3) x_{0} X^2 + 288 X Y 
\\
&\qquad
+ 24 (\phi^{15} + 39 \phi^{10} - 143 \phi^{5} - 3) x_{0}^3 X 
- 24 (\phi^{10} + 66 \phi^{5} - 11) x_{0}^2 Y 
\\
&\hspace{6.5em}
- (\phi^{25} + 363 \phi^{20} + 6446 \phi^{15} + 2982 \phi^{10} + 649 \phi^{5} - 9) x_{0}^5,
\\
f_{2}&=
-144 (3 \phi^{5} - 1) x_{0} X^2 - 288 X Y 
\\
&\qquad
- 24 (3 \phi^{15} - 143 \phi^{10} - 39 \phi^{5} + 1) x_{0}^3 X 
- 24 (11 \phi^{10} + 66 \phi^{5} - 1) x_{0}^2 Y 
\\
&\hspace{6.5em}
+ (9 \phi^{25} + 649 \phi^{20} - 2982 \phi^{15} + 6446 \phi^{10} - 363 \phi^{5} + 1) x_{0}^5,
\\
d&=
1440 X^2 + 1200 (\phi^{10} + 1) x_{0}^2 X 
\\
&\hspace{8.5em}
+ 2 (89 \phi^{20} - 792 \phi^{15} - 4034 \phi^{10} + 792 \phi^{5} + 89) x_{0}^4.
\end{align*}

2)  When the base field is $\C$ and $\phi=\phi(\tau)$, we have the following identities for the coefficients of $W_\phi$:
\begin{align*}
&\phi^{20}-228 \phi^{15}+494 \phi^{10}+228 \phi^{5}+1=\frac{\eta(\tau)^{12}}{\theta_2(0)^{20}}\cdot E_4(\tau), \\
&\phi^{30}+522 \phi^{25}-10005 \phi^{20}-10005 \phi^{10}-522 \phi^{5}+1=\frac{\eta(\tau)^{18}}{\theta_2(0)^{30}}\cdot E_6(\tau),
\end{align*}
where 
\begin{align*}
E_4(\tau)&=1+240\sum_{n=1}^\infty \Bigl(\sum_{d|n}d^3\Bigr) q^n=1+240q+2160q^2+\cdots\\
\intertext{and}
E_6(\tau)&=1-504\sum_{n=1}^\infty \Bigl(\sum_{d|n}d^5\Bigr) q^n=1-504q-16632q^2-\cdots
\end{align*}
are the standard normalized Eisenstein series on the full modular group of weight 4 and 6 respectively, and
\[ \eta(\tau)=q^{\frac1{24}}\prod_{n=1}^\infty (1-q^n) \]
is the Dedekind eta function.
Hence, if we make the change of variables
\[ \frac{\theta_2(0)^{10}}{\eta(\tau)^{6}}X\to X,\quad \frac{\theta_2(0)^{15}}{\eta(\tau)^{9}}Y\to Y, \]
we see that $W_\phi$ is isomorphic to 
\[ Y^2=X^3-\frac{E_4(\tau)}{48}X+\frac{E_6(\tau)}{864}.  \]
The discriminant of this curve is 
\[  -16\biggl(4\Bigl(-\frac{E_4(\tau)}{48}\Bigr)^3 +27\Bigl( \frac{E_6(\tau)}{864}\Bigr)^2\biggr)
=\frac{E_4(\tau)^3-E_6(\tau)^2}{1728}=\eta(\tau)^{24}.  \]
\end{remark}

\begin{lemma}\label{lem:2-torsion}
The multiplication-by-$[-1]$ map on the curve \eqref{eq:Bianchi} is given by $[-1]:(x_{0}:x_{1}:x_{2}:x_{3}:x_{4})\mapsto (x_{0}:x_{4}:x_{3}:x_{2}:x_{1})$.  As a consequence $2$-torsion points other than $O$ satisfy $x_{1}=x_{4}$ and $x_{2}=x_{3}$.
\end{lemma}

We remark that the conditions $x_{1}=x_{4}$ and $x_{2}=x_{3}$ is also visible from the expression of the $Y$-coordinate in Proposition~\eqref{prop:BianchiWeier}.
 
\section{$2$-division points and Bring's curve}

The equations \eqref{eq:Bianchi} may be considered to determine a surface in $\P^{4}\times \P^{1}$.  It is nothing but the elliptic modular surface $S(5)$ (in the sense of Shioda).  The $0$-section is given by $\phi\mapsto (0:\phi:-1:1:-\phi)$.  Consider the subvariety $D\subset \P^{4}\times \P^{1}$ defined by $x_{1}-x_{4}=x_{2}-x_{3}=0$. The intersection $S(5)\cap D$ is a curve and, by construction, its normalization parametrizes elliptic curves with a level $5$ structure and a distinguished $2$-torsion point different from~$O$.  In other words it is the modular curve associated with the group $\Gamma_{0}(2)\cap \Gamma(5)$.  (See Hulek \cite{Hulek:1993}.)

Letting $x_{4}=x_{1}$ and $x_{3}=x_{2}$ in \eqref{eq:Bianchi}, we have
\begin{equation}\label{eq:2-torsion}
\left\{
\begin{aligned}
&\phi^2 x_{2}^2 + \phi x_{0}^2 - x_{1}^2=0, \\
&\phi^2 x_{1} x_{2} + \phi x_{1}^2 - x_{0} x_{2}=0, \\
&\phi^2 x_{0} x_{1} + \phi x_{2}^2 - x_{1} x_{2}=0.
\end{aligned}\right.
\end{equation}
Eliminating $\phi$ from \eqref{eq:2-torsion}, we obtain a plane curve in $\P^{2}$:
\begin{equation}\label{eq:Hulek-Craig}
x_{0}^4 x_{1} x_{2} - x_{0}^2 x_{1}^2 x_{2}^2 - x_{0} x_{1}^5 - x_{0} x_{2}^5 + 2 x_{1}^3 x_{2}^3 = 0.
\end{equation}
This is a curve of genus~$4$, and known to be birational over $\Q(\zeta_{5})$ to Bring's curve defined by the equations
\[
\sum_{k=0}^{4} x_{k} = \sum_{k=0}^{4} x_{k}^{2}
=\sum_{k=0}^{4} x_{k}^{3}=0.
\]
(See \cite[Th.~3]{Dye} and \cite[Th.~1]{Braden-Northover}).  The curve \eqref{eq:Hulek-Craig} is called the Hulek-Craig model of Bring's curve.  Bring’s curve is the unique curve of genus $4$ with automorphism group $S_{5}$, and has many exceptional properties.  Its relation to modular curves is also well studied (e.g. \cite{Braden-Disney-Hogg}\cite{Weber}\cite{Horie-Yamauchi}).

Eliminating $x_{0}$ from \eqref{eq:2-torsion}, we have another model of Bring's curve in $\P^{2}\times\P^{1}$:
\begin{equation}\label{eq:Bring2}
\phi^4 x_{1}^2 x_{2} + \phi^3 x_{1}^3 + \phi x_{2}^{3} - x_{1} x_{2}^{2} = 0.
\end{equation}
Furthermore, substituting $\xi=\phi x_{2}/x_{1}$, we obtain yet another affine plane curve model of Bring's curve:
\begin{equation}\label{eq:KKmodel}
\xi^{3} - \xi^{2} + \phi^{5}\xi + \phi^{5}= 0.
\end{equation}
The roots of this equation, viewed as a cubic equation in $\xi$ over $\Q(\phi)$, supply coordinates of the $2$-torsion points
of the elliptic curve $E_\phi$.

\begin{theorem}\label{thm:2tors}
Let $g_i\ (i=1,2,3)$ be the roots of the cubic equation~\eqref{eq:KKmodel} in $\xi$.
Then the $2$-torsion points of $E_\phi$ other than the origin are
\begin{equation}\label{2tors}  [\phi^3+\phi^3/g_i:\phi:g_i:g_i:\phi]
\quad (i=1,2,3). \end{equation}
If $\phi=\phi(\tau)$, then the $g_i$'s are the modular functions given by
\begin{align*}
g_1(\tau)&=\frac{\phi(\tau)^2}{\phi(2\tau)}=1-2q+4q^2-4q^3+2q^4+2q^5-8q^6+\cdots,\\
g_2(\tau)&=-\phi(\tau/2)\phi(\tau)^2=-q^{1/2}+q+q^{3/2}-2q^2+2q^3-2q^{7/2}-\cdots,\\
g_3(\tau)&=\frac{\phi(\tau)\phi(2\tau)}{\phi(\tau/2)}=q^{1/2}+q-q^{3/2}-2q^2+2q^3+2q^{7/2}-\cdots.
\end{align*}
\end{theorem}
\begin{proof} That the points~\eqref{2tors} have order two is readily seen from the addition formulas 
in \S2 and the equation~\eqref{eq:KKmodel} that the $g_i$'s satisfy.
In the case of $\phi=\phi(\tau)$, let us {\em define} $g_i(\tau)$ as given in the theorem. Then, by using the 
``modular equations'' of Ramanujan~\cite{RamanujanLost1}, we can show the identities
\[\setlength{\arraycolsep}{2pt}\renewcommand{\arraystretch}{1.3}
\begin{array}{ccc} 
g_1(\tau)+g_2(\tau)+g_3(\tau) &=&1,\\
g_1(\tau)g_2(\tau)+g_2(\tau)g_3(\tau)+g_3(\tau)g_1(\tau) &=&\phi(\tau)^5,\\
g_1(\tau)g_2(\tau)g_3(\tau) &=&-\phi(\tau)^5
\end{array}\]
hold, which means the $g_i(\tau)$'s are the roots of~\eqref{eq:KKmodel} with $\phi=\phi(\tau)$.
Specifically, we use equations~(1.5.1) and (1.5.3) in~\cite[p.~24--25]{RamanujanLost1} to show the first 
identity and (1.5.1) for the second, whereas the third one is immediate from the definition.
\end{proof}

\begin{remark}
1) The model \eqref{eq:KKmodel} is isomorphic with Weber's model 
\[
y^{5} = (x+1)x^{2}(x-1)^{-1}
\]
(\cite[Prop~3.1]{Weber}) through the change of variables $x=-\xi$, $y=-\phi$.

2)  The function
\[ -g_2(2\tau)=\phi(\tau)\phi(2\tau)^2=q-q^2-q^3+2q^4-2q^6+2q^7-\cdots,\]
is one of the ``Ramanujan's functions.''  He discovered the relation 
\begin{equation}\label{eq:Rama}  -g_2(2\tau)=\frac{1-g_1(\tau)}{1+g_1(\tau)}, \end{equation}
which is equivalent to (1.5.3) in~\cite{RamanujanLost1}.

3) The discriminant of the polynomial $x^3-x^2+\phi^5 x+\phi^5$ is 
\begin{align*}
&4\phi^5(1-11\phi^5-\phi^{10}) \\
&= -4\phi^5(\phi^2+\phi-1)(\phi^4-3\phi^3+4\phi^2-2\phi+1)(\phi^4+2\phi^3+4\phi^2+3\phi+1).
\end{align*}
This has the same roots as the discriminant of $E_\phi$. Hence, if the elliptic curve $E_\phi$ is
non-singular, the roots $g_i$ are always distinct and give all non-trivial $2$-torsions.
\end{remark}

Using these functions, we can describe several modular function fields of level $10$ in between
$\Gamma(10)$ and $\Gamma_1(5)$ and give their defining equations in a rather uniform way.

\section{Various modular function fields of level 10}

First we display the diagram that depicts the inclusions of function fields between $A_0(\Gamma_0(5))$ 
and $A_0(\Gamma(10))$, where for a congruence subgroup $G$, we denote by $A_0(G)$ the modular
function field over $\C$ associated to $G$.  The symbol such as $10\mathrm{A}^5$ is the name of the group in the
database~\cite{Congruence}, the number in front indicates the level and the superscript is the genus.
The line between the fields means the extension of the fields, and the number along side the line is
the degree of the extension.

\begin{figure}
\includegraphics[scale=0.8]{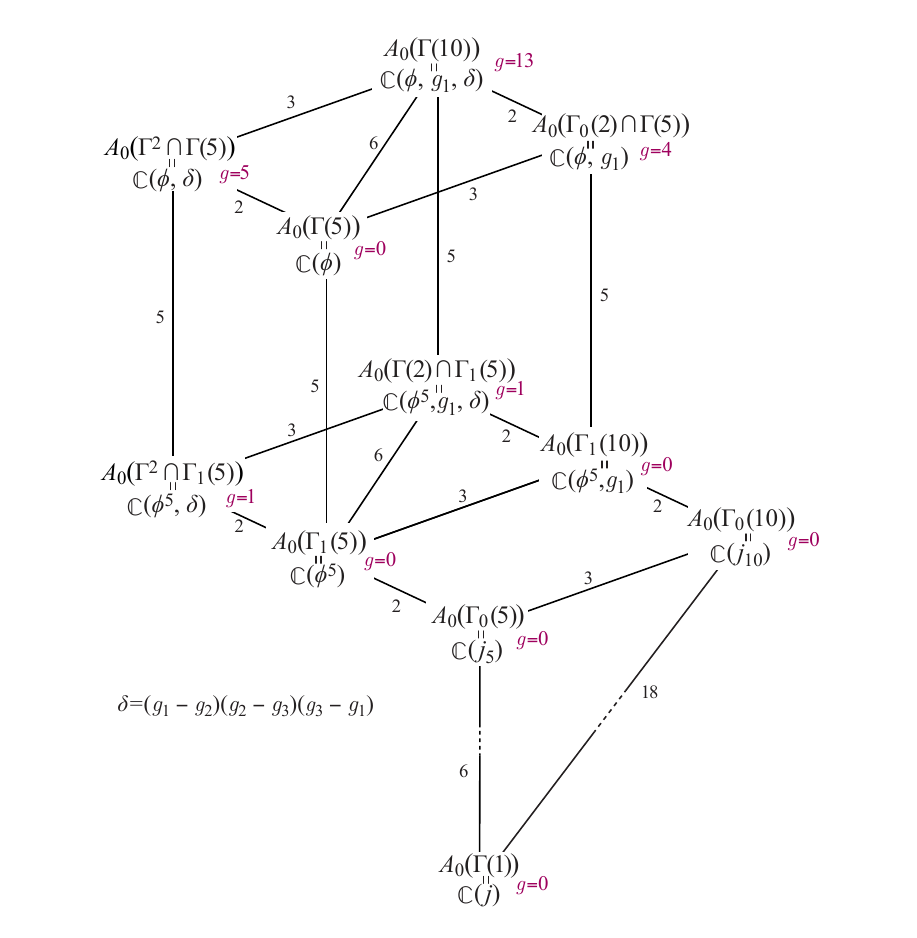}
\caption{Function fields between $A_0(\Gamma_0(5))$ and $A_0(\Gamma(10))$.}
\label{fig:function-fields}
\end{figure}

\subsection{Principal congruence subgroup of level $10$} 

We start with the largest field in the diagram, i.e., the function field of the principal congruence subgroup $\Gamma(10)$.
This group has genus $13$, and its function field and the defining equation are given for instance in~\cite{Ishida} and~\cite{Yifan}.  
Here we give another descriptions of them in terms of our functions
$\phi(\tau)$ and $g_i(\tau)$. Let $\delta(\tau)$ be the discriminant of the equation~\eqref{eq:KKmodel}:
\[ 
\delta(\tau):=(g_1(\tau)-g_2(\tau))(g_2(\tau)-g_3(\tau))(g_3(\tau)-g_1(\tau)). \]

\begin{theorem}  The function field $A_0(\Gamma(10))$ of the principal congruence subgroup $\Gamma(10)$ is given by
\[ 
A_0(\Gamma(10))\!=\!\C\bigl(\phi(\tau), g_1(\tau), \delta(\tau)\bigr)
\!=\!\C\bigl(\phi(\tau), g_1(\tau),g_2(\tau)\bigr)
\!=\!\C\Bigl(\phi(\tau), \frac{\delta(\tau)}{g_1(\tau)}\Bigr). 
\]
The defining equation satisfied by $X=\phi(\tau)$ and $Y=\delta(\tau)/2g_1(\tau)$ 
\begin{align}\label{eq:defeq10}  &Y^2(1-11X^5-X^{10}-Y^2)^2\\
&\qquad -X^5(1-11X^5-X^{10})(1-11X^5-X^{10}+Y^2)^2 = 0, \nonumber
\end{align}
which is of genus $13$.

Moreover, the function $g_1(\tau)$ can be expressed in terms of $X$ and $Y$ as
\begin{equation}\label{eq:g1byXY}  g_1(\tau)=\frac{1-11X^5-X^{10}-Y^2}{1-11X^5-X^{10}+Y^2}. \end{equation}
\end{theorem}

\begin{proof}
Since we know that the $\phi(\tau)$ is a generator of $A_0(\Gamma(5))$ ({\it cf.} \cite{Duke}), we see from the formulas in Theorem~\ref{thm:2tors}
for $g_i(\tau)$ by $\phi(\tau)$ that all $\phi(\tau), g_i(\tau), \delta(\tau)$ belong to $A_0(\Gamma(10))$. Since the quotient group $\Gamma(5)/\Gamma(10)$
is isomorphic to the symmetric group $S_3$, we can conclude that $A_0(\Gamma(10))$ is obtained from $A_0(\Gamma(5))$
by adjoining all the roots $g_1(\tau), g_2(\tau), g_3(\tau)$ of ~\eqref{eq:KKmodel}, or one root and the discriminant:
\[ 
A_0(\Gamma(10))=\C\bigl(\phi(\tau), g_1(\tau),g_2(\tau)\bigr)
=\C\bigl(\phi(\tau), g_1(\tau), \delta(\tau)\bigr). 
\]
(Note the relation $g_1(\tau)+g_2(\tau)+g_3(\tau) =1$.) To show that this field is generated by the two elements $X=\phi(\tau)$ and $Y=\delta(\tau)/2g_1(\tau)$,
it is sufficient to show the identity~\eqref{eq:g1byXY}.  First, the square of the discriminant of~\eqref{eq:KKmodel} can be written as a polynomial of the coefficients as
\begin{equation}\label{eq:delsq}  
\delta(\tau)^2=4\phi(\tau)^5\bigl(1-11\phi(\tau)^5-\phi(\tau)^{10}\bigr),  
\end{equation}
which is equivalent to 
\begin{equation}\label{eq:delsq2} Y^2=X^5(1-11X^5-X^{10})/g_1(\tau)^2. \end{equation}
Inserting this into the right-hand side of~\eqref{eq:g1byXY}, we see that it is equal to 
\[ \frac{1-X^5/g_1(\tau)^2}{1+X^5/g_1(\tau)^2} = \frac{g_1(\tau)^2-\phi(\tau)^5}{g_1(\tau)^2+\phi(\tau)^5}, \]
which is identical to $g_1(\tau)$ since $g_1(\tau)$ is a root of~\eqref{eq:KKmodel}.

The identity~\eqref{eq:defeq10} is now immediate from~\eqref{eq:g1byXY} squared and~\eqref{eq:delsq2}. 
We can check by the {\tt algcurves} package in Maple that the curve~\eqref{eq:defeq10} has genus~$13$. 
\end{proof}

From the next subsection, we describe various sub-fields of~$A_0(\Gamma(10))$.

\subsection{Genus 0 groups}

In the diagram Fig.~\ref{fig:function-fields}, there are five groups which have genus 0 other than $\Gamma(1)=\mg$, in which case
the generator of $A_0(\Gamma(1))$ is the famous elliptic modular $j$-function
\[ j(\tau)=\frac1q+744+196884\thin q+21493760\thin q^2+\cdots. \] 
The generators of all those five function fields are more or less well known, possibly except for the 
group $\Gamma_1(10)$. We just list the generators. See for instance~\cite{Lee-Park} and~\cite{Duke}.
\[\setlength{\arraycolsep}{2pt}\renewcommand{\arraystretch}{2.2}
\begin{array}{lcrcl}
\Gamma_0(5)&:& j_5(\tau)&:=&\dfrac{\eta(\tau)^6}{\eta(5\tau)^6}=\dfrac1q-6+9\thin q+10\thin q^2-30\thin q^3+\cdots,\\
\Gamma_1(5)&:& \phi(\tau)^5&=& q-5\thin q^2+15\thin q^3-30\thin q^4+40\thin q^5-\cdots,\\
\Gamma_0(10)&:&\  j_{10}(\tau)&:=&\dfrac{\eta(2\tau)\eta(5\tau)^5}{\eta(\tau)\eta(10\tau)^5}=\dfrac1q+1+q+2\thin q^2+2\thin q^3-\cdots,\\
\Gamma_1(10)&:& g_1(\tau)&=&\dfrac{\phi(\tau)^2}{\phi(2\tau)}=1-2\thin q+4\thin q^2-4\thin q^3+2\thin q^4+\cdots,\\
\Gamma(5)&:& \phi(\tau)&=& q^{1/5}-q^{6/5}+q^{11/5}-q^{21/5}+q^{26/5}-q^{31/5}+\cdots\\[-2ex]
&&&=& q^{1/5}(1-q+q^2-q^4+q^5-q^6+\cdots).
\end{array}
\]
Here, $\eta(\tau)=q^{1/24}\prod_{n=1}^\infty (1-q^n)$ is the Dedekind eta function.
For $\Gamma_1(10)$, it is proved in \cite[Proposition~2.1]{Lee-Park} that the function $g_2(2\tau)$
is a generator of $A_0(\Gamma_1(10))$, from which follows that $g_1(\tau)$ is also a generator
because of Ramanujan's relation~\eqref{eq:Rama}.

Relations among $ j_5(\tau)$, $\phi(\tau)^5$, and $ j_{10}(\tau)$ are classical (see~\cite{Duke, Lee-Park}):
\[ j_5(\tau)=\frac1{\phi(\tau)^5}-11-\phi(\tau)^5=\frac{(j_{10}(\tau)+1)(j_{10}(\tau)-4)^2}{j_{10}(\tau)^2}. \]
Also, by using (1.5) and (1.6) in~\cite{Lee-Park} and~\eqref{eq:Rama}, we can deduce
\[ j_{10}(\tau)=g_2(2\tau)-\frac1{g_2(2\tau)}=\frac{4g_1(\tau)}{1-g_1(\tau)^2}. \]

\subsection{Genus 1 groups}
There are two genus 1 groups in the diagram Fig.~\ref{fig:function-fields}, namely $10\mathrm{K}^{1}=\Gamma(2)\cap \Gamma_{1}(5)$ and  $10\mathrm{D}^1$.

\begin{proposition}  Let $G_1$ be the congruence subgroup defined by
\[ 
G_1:=\left\{\left.\sltwo(a,b;c,d) \in SL_{2}(\Z)\,\right|\, 
\begin{array}{l}
a\equiv d\equiv 1\, (\bmod\thin 10), \ b\equiv 0\, (\bmod\thin{2})\\
c\equiv 0\, (\bmod\thin10) 
\end{array}\!
\right\}. 
\]
This is a normal subgroup of $\Gamma_1(5)$ of index $6$, contains $\Gamma(10)$ as a subgroup of index $5$, and is of genus $1$.
Its function field $A_0(G_1)$ is generated by $g_1(\tau)$ and $g_2(\tau)$: 
\[ 
A_0(G_1)=\C\bigl(g_1(\tau), g_2(\tau)\bigr). 
\]
The equation satisfied by $X=g_1(\tau)$ and $Y=g_2(\tau)$ is 
\begin{equation}\label{eq:g1g2}  
X^2Y+XY^2+X^2+Y^2-X-Y=0. 
\end{equation}
This is isomorphic to 
\begin{equation}\label{eq:g1g2Weier}  
Y^2=X^3+X^2-X, 
\end{equation}
which is an elliptic curve of conductor $20$.
\end{proposition}

\begin{proof}  The description concerning the group is easy to check.  In particular, we see that the quotient group $\Gamma_1(5)/G_1$ is isomorphic to
the symmetric group $S_3$.  We can then conclude that 
\[ 
A_0(G_1)=\C\bigl(\phi(\tau)^5, g_1(\tau), \delta(\tau)\bigr)
=\C\bigl(g_1(\tau), g_2(\tau), g_3(\tau)\bigr) 
=\C\bigr(g_1(\tau), g_2(\tau)\bigr), 
\]
paralleling to the case of $A_0(\Gamma(10))$. Here, we have used the relations
\[ \phi(\tau)^5=\frac{g_1(\tau)^2-g_1(\tau)^3}{1+g_1(\tau)}\quad\text{and}\quad g_1(\tau)+g_2(\tau)+g_3(\tau) =1.\]
The relation between $g_1(\tau)$ and $g_2(\tau)$ is obtained from 
\[ g_1(\tau)g_2(\tau)+g_2(\tau)g_3(\tau)+g_3(\tau)g_1(\tau) = -g_1(\tau)g_2(\tau)g_3(\tau) \ (=\phi(\tau)^5) \]
by eliminating $g_3(\tau) = 1-g_1(\tau)- g_2(\tau)$. Any standard algorithm may be used to transform equation~\eqref{eq:g1g2} into the Weierstrass form~\eqref{eq:g1g2Weier}.  In fact, $X$ and $Y$ in \eqref{eq:g1g2Weier} are given by
\[
X = \frac{2 - g_{1}(\tau) - g_{2}(\tau)}{g_{1}(\tau) + g_{2}(\tau)}, \quad
Y = \frac{(g_{1}(\tau) - g_{2}(\tau))(2 - g_{1}(\tau) - g_{2}(\tau))}
{(g_{1}(\tau) + g_{2}(\tau))^2}.
\qedhere
\]
\end{proof}

\begin{remark}  
1) The group $G_1$ is equal to the intersection $\Gamma(2)\cap \Gamma_{1}(5)$. Note that, if integers $a,b,c,d$ satisfy $ad-bc=1$ 
and $bc$ is even, then $a$ and $d$ are both odd and thus the congruence $a\equiv d\equiv 1\, (\bmod\, 5)$ implies $a\equiv d\equiv 1\, (\bmod\, 10)$.

2)  The elliptic curve \eqref{eq:g1g2Weier} is labeled $20.a3$ in the LMFDB~\cite{LMFDB}.
\end{remark}

\begin{proposition}\label{prop:G1}  
%Let $\Gamma^{2}$ be the congruence subgroup defined by
%\[ 
%\Gamma^{2}:=\left\{\!\left.\sltwo(a,b;c,d) \in SL_{2}(\Z) \right|
%\sltwo(a,b;c,d)\!\equiv\!\sltwo(1,0;0,1),\sltwo(1,1;1,0),\sltwo(0,1;1,1) (\bmod 2)\!\right\},
%\]
Let $G_{2}:=\Gamma^{2}\cap \Gamma_{1}(5)$.
The group $G_{2}$ is a subgroup of $\Gamma_1(5)$ of index $2$, contains $G_1$ as a subgroup of index $3$, and is of genus $1$.
Its function field $A_0(G_2)$ is generated by $\phi(\tau)^5$ and the discriminant $\delta(\tau)$,
\[
A_0(G_2)=\C\bigl(\phi(\tau)^5, \delta(\tau)\bigr), 
\]
and the defining equation satisfied by $X=-\phi(\tau)^5$ and $Y=\delta(\tau)/2$ is given by
\begin{equation}\label{eq:10D1}  
Y^2=X^3-11X^2-X. 
\end{equation}
\end{proposition}

\begin{proof}    Again, the descriptions of the group $G_2$ is easily verified. From this, we conclude that the function field $A_0(G_2)$
is the unique intermediate field between $A_0(\Gamma_1(5))=\C\bigl(\phi(\tau)^5\bigr)$ and $A_0(G_1)=\C\bigl(\phi(\tau)^5, g_1(\tau), \delta(\tau)\bigr)$ that
corresponds Galois theoretically to the subgroup $G_2/G_1$ of order $3$ of the group $\Gamma_1(5)/G_1\simeq S_3$. 
We therefore conclude $A_0(G_2)=\C\bigl(\phi(\tau)^5, \delta(\tau)\bigr)$.  The defining equation is nothing but~\eqref{eq:delsq}.
\end{proof}

\begin{remark}  1) The group $G_2$ is labeled as $10\mathrm{D}^1$ in the database~\cite{Congruence}.

2)  The minimal model of the above elliptic curve is
\[ 
Y^2=X^3+X^2-41X-116, 
\]
which has conductor $20$ and is labeled $20.a1$ in the LMFDB~\cite{LMFDB}.
\end{remark}

\subsection{Genus 4 group (Bring's curve)}

\begin{proposition}  Let $G_3$ be the congruence subgroup defined by
\[ 
G_3:=\left\{\left.\sltwo(a,b;c,d) \in SL_{2}(\Z)\,\right|\, 
\begin{array}{l}
a\equiv d\equiv 1\, (\bmod\thin 10), \ b\equiv 0\, (\bmod\thin{5})\\
c\equiv 0\, (\bmod\thin 10) 
\end{array}\!
\right\}. 
\]
This is equal to $\Gamma_0(2)\cap\Gamma(5)$ and is a subgroup of $\Gamma(5)$ of index $3$, contains $\Gamma(10)$ as a subgroup of index $2$, and is of genus $4$.
Its function field $A_0(G_3)$ is generated by $\phi(\tau)$ and $g_1(\tau)$: 
\[ 
A_0(G_3)=\C\bigl(\phi(\tau), g_1(\tau)\bigr). 
\]
The equation satisfied by $X=\phi(\tau)$ and $Y=g_1(\tau)$ is our cubic equation
\begin{equation}\label{eq:BringKK} Y^3-Y^2+X^5\thin Y+X^5=0. \end{equation}
\end{proposition}

\begin{proof}  As before, we leave the readers for the verification of the descriptions concerning groups.

The function field $A_0(G_3)$ is then an extension of $A_0(\Gamma(5))=\C\bigl(\phi(\tau)\bigr)$ of degree $3$. By $g_1(\tau)=\phi(\tau)^2/\phi(2\tau)$, the function
  $g_1(\tau)$ is an element of $A_0(G_3)$ of degree $3$ over $\C\bigl(\phi(\tau)\bigr)$. Therefore we conclude $A_0(G_3)=\C\bigl(\phi(\tau), g_1(\tau)\bigr)$, and the
  defining equation is given by~\eqref{eq:KKmodel}.
\end{proof}

\begin{remark}  1) The group $G_3$ is labeled as $10\mathrm{B}^4$ in~\cite{Congruence}.

2)  The curve~\eqref{eq:BringKK} is a model of Bring's curve, as discussed in \S4.
\end{remark}

\subsection{Genus 5 group}

\begin{proposition}  Let $G_4:=\Gamma^{2}\cap \Gamma(5)$. 
% be the congruence subgroup defined by
%\[ 
%G_4:=\left\{\!\left.\sltwo(a,b;c,d) \in SL_{2}(\Z) \right|
%\sltwo(a,b;c,d)\!\equiv\!\sltwo(1,0;0,1),\sltwo(1,5;5,6),\sltwo(6,5;5,1) (\bmod 10)\!\right\}. 
%\]
The group $G_4$ is a subgroup of $\Gamma(5)$ of index $2$, contains $\Gamma(10)$ as a subgroup of index $3$, and is of genus $5$.
Its function field $A_0(G_4)$ is generated by $\phi(\tau)$ and $\delta(\tau)$: 
\[ 
A_0(G_3)=\C\bigl(\phi(\tau), \delta(\tau)\bigr). 
\]
The equation satisfied by $X=\phi(\tau)$ and $Y=\delta(\tau)/2$ is 
\[ 
Y^2=X^5(1-11X^5-X^{10}), 
\]
which, by changing $Y/X^2\to iY$, is birationally equivalent to
\begin{equation}\label{eq:genus5}
Y^2=X(X^{10}+11X^5-1). 
\end{equation}
\end{proposition}

\begin{proof}
That the set $G_4$ is indeed a group and is a subgroup of $\Gamma(5)$ of index $2$ is easy to verify. Since $\Gamma(5)/\Gamma(10)$ is isomorphic
to $S_3$, the function field $A_0(G_4)$ corresponds Galois theoretically to the unique extension of $A_0(\Gamma(5))=\C\bigl(\phi(\tau)\bigr)$ of degree $2$,
which is $\C\bigl(\phi(\tau), \delta(\tau)\bigr)$. The equation satisfied by $X=\phi(\tau)$ and $Y=\delta(\tau)/2$ is given in Proposition~\ref{prop:G1}.
\end{proof}

\begin{remark}\label{rmk:genus5}
  1) The group $G_4$ is labeled as $10\mathrm{A}^5$ in~\cite{Congruence}.

2)  The curve~\eqref{eq:genus5} is a hyperelliptic curve of genus~$5$ with automorphism group $C_2\times A_5$.  See ``Heigher genus famillies'' section of LMFDB\cite{LMFDB}.

\end{remark}

\subsection*{Acknowledgements}
Kaneko was supported by JSPS KAKENHI Grant Numbers JP21K18141 and JP21H04430. 
Kuwata was supported by the Chuo University Grant for Special Research.

%\bibliography{KKlevel10.bib}
%\bibliographystyle{amsplain}
\providecommand{\bysame}{\leavevmode\hbox to3em{\hrulefill}\thinspace}
\def\MR#1{\relax}

\end{document}